\documentclass[11pt]{article}
\usepackage{amsmath, amssymb, amsfonts, amstext, amsthm, textcomp, enumerate}
\usepackage[mathscr]{euscript}
\usepackage{float}
\usepackage{booktabs}
\usepackage{mathtools}
\usepackage[left=20mm,top=0.5in,bottom=15mm]{geometry}
\usepackage{graphicx}
\usepackage{caption}
\usepackage{epstopdf}
\usepackage{longtable}
\usepackage[utf8]{inputenc}
\usepackage{xcolor}
\usepackage{hyperref}
\usepackage{graphicx}
\usepackage{dcolumn}
\usepackage{bm}
\usepackage{epstopdf}
\usepackage[english]{babel}
\usepackage{subfigure}
\usepackage{color}
\usepackage{ulem}

\newtheorem{thm}{Theorem}[section]
\newtheorem{lem}[thm]{Lemma}

\newtheorem{ex}[thm]{Example}

\newfont{\bb}{msbm10}

\begin{document}
	\title{ Determinant Bounds for $(n-1)$-Locally Positive Semidefinite Matrices}
      \author{Shaun Fallat\thanks{Department of Mathematics and Statistics,
University of Regina, Regina, SK, Canada 
(sfallat@uregina.ca).}
    \and  Samir Mondal\thanks{Department of Mathematics and Statistics,
University of Regina, Regina, SK, Canada (isamirmondal@gmail.com)}
  \and Hristo Sendov\thanks{Department of Statistical and Actuarial Sciences, Department of Mathematics, The University of Western Ontario, London, Ontario, Canada (hsendov@uwo.ca)} 
  }
  \date{}
 \maketitle
\begin{abstract}
Semidefinite programming relaxations based on $k$-locally positive semidefinite constraints provide tractable approximations to many nonconvex discrete and continuous optimization problems. In this framework, the extremal case $k = n-1$ corresponds to the tightest nontrivial relaxation in this hierarchy, in which every principal submatrix of size $n-1$ is constrained to be positive semidefinite, while the global positive semidefiniteness condition is governed by the determinant.  

In this paper, we study the determinants of the $(n-1)$-locally positive semidefinite matrices and derive sharp lower bounds on their determinants that quantify the gap between local and global positive semidefiniteness. Our main result establishes an extension of Hadamard’s determinantal inequality, showing that for any such matrix $A=(a_{ij})$,
\[
\det(A) \ge -\frac{1}{n-2}\left(\frac{n-1}{n-2}\right)^{n-1} a_{11}\cdots a_{nn}.
\]
We further obtain analogous extensions of classical determinant inequalities, including Fisher’s and Koteljanskii’s inequalities, providing tight lower bounds in each case.  In an sense, these results quantify, via determinant bounds, how far the class of $(n-1)$-locally positive
semidefinite matrices can be from being positive semidefinite. In particular, we show
that although global positive semidefiniteness may fail, the determinant is tightly controlled, offering new insight into the quality of such relaxations.

\end{abstract}


\noindent AMS-MSC: 15A15, 15A18; 15B48

\vspace{.5cm}
\section{Introduction} 

Semidefinite programming (SDP) relaxations are a fundamental tool for obtaining tractable approximations and dual bounds for a wide range of discrete and continuous non-convex optimization problems \cite{wolkowicz2012handbook}. Many applications in areas such as power systems, control, and signal processing can be formulated or relaxed as an SDP, where the decision variable is constrained to lie in the cone of symmetric positive semidefinite (PSD) matrices
\[
S_n^+: = \{X \in \mathcal{S}^n \mid x^T X x \geq 0,\; \forall x \in \mathbb{R}^n\}.
\]
{Here $\mathcal{S}^n$ denotes the set of $n \times n$ symmetric matrices.
The cone of symmetric positive definite (PD) matrices is
\[
S_n^{++} := \{X \in \mathcal{S}^n \mid x^T X x > 0,\; \forall x \in \mathbb{R}^n, x \not=0\}.
\]}
Although SDP problems are indeed solvable in polynomial time, enforcing the global PSD constraint $X \in S_n^+$ becomes computationally challenging for large-scale problems. To address this issue, a widely used approach is to relax the global constraint by requiring only certain principal submatrices be PSD. This leads to the family of cones
\begin{align*}
S^+_{n,k} &:= \{X \in \mathcal{S}^n \mid \text{all } k \times k \text{ principal submatrices of } X \text{ are PSD}\}, 
\end{align*}
which form a hierarchy satisfying $S_n^+ \subseteq S^+_{n,k}$, with equality when $k = n$.  These relaxations are computationally attractive and have been successfully applied in practice to problems such as quadratically constrained quadratic programs and optimal power flow \cite{baltean2018selecting, dey2019sparse, qualizza2012linear, kocuk2016strong, sojoudi2014exactness}. 

Understanding the quality of these relaxations is therefore of both practical and theoretical interest. Prior work has studied the cones $S^+_{n,k}$, for various $k$, from both a geometric and a spectral perspective. In particular, \cite{blekherman2022sparse} analyzes their distance to the PSD cone, showing that the approximation improves as $k$ increases and becomes exact when $k=n$, while for small values of $k$ the distance can remain bounded away from zero. Complementarily, \cite{Blekherman2022} shows that the eigenvalues of matrices in $S^+_{n,k}$ lie in the convex cone $H(e_k^n)$, the hyperbolicity cone of the elementary symmetric polynomial $e_k^n.$

In this paper, we study the extremal case $k=n-1$, which represents the tightest nontrivial relaxation in this hierarchy. In this regime, every principal minor of a matrix is constrained except possibly the determinant, so the gap between $S^+_{n,n-1}$ and $S_n^+$ is entirely governed by the determinant. This observation raises a natural question: How large can the violation of the global positivity condition be, given that all local PSD constraints are satisfied? We resolve this question by establishing a quantitative lower bound on the determinant, see Theorem~\ref{2026-04-16-thm1}

For our purposes, we need to modify the definition of the set $S^+_{n,n-1}$. So, for the remainder of the paper, we define
{
\begin{align*}
S^+_{n,n-1} &:= \{X \in \mathcal{S}^n \mid \text{all principal submatrices of size $n-1$ are PSD and } \det(X) < 0\}.
\end{align*}
Let $S^{++}_{n,n-1}$ be the analogous set with PSD replaced by PD.}
We refer to $S^+_{n,n-1}$ as the cone of $(n-1)$-locally positive semidefinite matrices and 
to $S^{++}_{n,n-1}$ as the cone $(n-1)$-locally positive definite matrices.

We are also motivated by the following classical determinant inequalities, see Section~\ref{sec:notation} for notation.
For any matrix $A= (a_{ij}) \in S_n^+$, the following inequalities hold:
\begin{itemize}
    \item \textbf{Hadamard's inequality:} 
    \begin{equation}
\label{eq:hadamard}
\det(A)\leq a_{11}\cdots a_{nn}.
\end{equation}
    \item \textbf{Fisher's inequality:} for any $\alpha \subseteq \{1,\ldots,n\}$,
    \begin{equation}
    \label{eq:fisher}
    \det(A)\leq \det(A[\alpha]) \det(A[\alpha^c]).
    \end{equation}
    
    \item \textbf{Koteljanskii's inequality:} for any $\alpha,\beta \subseteq \{1,\ldots,n\}$,
    \begin{equation}
    \label{eq:koteljanskii}
    \det(A[\alpha \cup \beta]) \det(A[\alpha \cap \beta])
    \leq \det(A[\alpha]) \det(A[\beta]).
    \end{equation}
\end{itemize}

Our main results give a sharp lower bound on the determinant of matrices in $S^+_{n,n-1}$. 

\begin{thm}[Extended Hadamard's inequality]
\label{2026-04-16-thm1}
Let $A \in S^+_{n,n-1}$, with $n \ge 3$. Then
\[
\det(A)\geq -\frac{1}{n-2}\left(\frac{n-1}{n-2}\right)^{n-1}a_{11}\cdots a_{nn}.
\]
{The inequality is strict when $A \in S^{++}_{n,n-1}$.} 
\end{thm}

\begin{thm}[Extended Fisher's inequality]\label{2026-04-04-fisher}
Let $A \in S^+_{n,n-1}$, with $n \ge 3$, and let $\alpha \subseteq \{1,\ldots,n\}$. Then
\[
\det(A) \ge -\frac{1}{n-2}\left(\frac{n-1}{n-2}\right)^{n-1}
\det (A[\alpha])\det (A[\alpha^c]).
\]
{The inequality is strict when $A \in S^{++}_{n,n-1}$.} 
\end{thm}

\begin{thm}[Extended Koteljanskii's inequality]\label{koteljanskii}
Let $A \in \mathcal{S}^n$, and let $\alpha, \beta \subseteq \{1,\dots,n\}$, with $r:=|(\alpha \cup \beta) \setminus (\alpha \cap \beta)| \ge 3$ {and $m:=|\alpha \cup \beta|$. If $A[\alpha \cup \beta] \in S^+_{m,m-1}$}, then
\begin{equation*}
\label{eq:ext-koteljanskii}
\det (A[\alpha \cup \beta]) \det (A[\alpha \cap \beta])
\ge 
-\frac{1}{r-2}\left(\frac{r-1}{r-2}\right)^{r-1}
\det (A[\alpha]) \det (A[\beta]).
\end{equation*}
{The inequality is strict when $A[\alpha \cup \beta] \in S^{++}_{m,m-1}$.} 
\end{thm}

In addition, in the following theorem, we give a lower bound on the determinant of matrices in a related class defined by positive semidefinite constraints on their leading principal submatrices. Such matrices have long been studied. For instance, in the symmetric case, positivity of all leading principal minors characterizes positive definiteness. More recently, matrices with prescribed constraints on leading principal minors have appeared in the context of Cholesky factorization (see \cite{KV}).

\begin{thm}
Let
\begin{align*}
A=\begin{pmatrix}
B & b\\
b^T & a_{nn}
\end{pmatrix}\in\mathcal S^n~~
\mbox{with $n\geq 3$}, 
\end{align*}
where $B$ is an $(n-1) \times (n-1)$ positive semidefinite matrix. Assume that
\begin{align}\label{conditions}
{\det(A) \le 0, \,\,\, a_{nn} \ge 0, \text{ and } \,\,
|b_i|\le \sqrt{a_{ii}a_{nn}},\quad i=1,\dots,n-1.}
\end{align}
Then
\begin{align*}
\det(A)\ge
-(n-1)\left(\frac{n-1}{n-2}\right)^{n-2}\,a_{11}\cdots a_{nn}.
\end{align*}
The inequality is strict when $B$ is positive definite and \eqref{conditions} holds with $a_{nn} >0$.
\end{thm}
These results provide a determinant-based perspective on the gap between local and global PSD constraints. In particular, they show that although $(n-1)$-local PSD constraints does not necessarily guarantee positive semidefiniteness, the violation of classical determinant inequalities is tightly controlled.
In addition, we hope that our results contribute to and enrich the extensive literature on determinantal inequalities, much of which focuses on positive semidefinite and related classes of matrices. This literature has a long history beginning with the classical works of Hadamard~\cite{Had} and Fischer~\cite{Fis}, and subsequent important developments by Carlson~\cite{C1,C2}, Koteljanskii~\cite{Ko1, Ko2}, and others; see also the surveys and related treatments in \cite{MM, BJ2,Bhatia1997,FJ}.

\section{Notation}
\label{sec:notation}

This work involves real $n \times n$ matrices, denoted by $\mathbb{R}^{n \times n}$.  We let
$I_m$ and $J_m$ denote the $m\times m$ identity matrix and all-ones matrix, and we use the notation $\mathbf{1_n} \in \mathbb{R}^{n}$ for the all-ones vector. We may omit the subscripts if the size is evident from the context.

Let $N := \{1,\ldots,n\}$. For any $A \in \mathbb{R}^{n \times n}$ and index sets $\alpha, \beta \subseteq N$, we denote by $A[\alpha, \beta]$ the submatrix of $A$ whose rows and columns are indexed by $\alpha$ and $\beta$, respectively, where the indices are taken in ascending order. When $\alpha = \beta$, we write $A[\alpha] := A[\alpha, \alpha]$ and refer to it as a {principal submatrix} of $A$. The quantity $\det(A[\alpha])$ is called a {principal minor} of order $|\alpha|$. By convention, $A[\emptyset] := 1$.

For $\alpha, \beta \subseteq N$, we use $A(\alpha, \beta)$ to denote the submatrix obtained from $A$ by deleting the rows indexed by $\alpha$ and the columns indexed by $\beta$. We write $A(\alpha) := A(\alpha, \alpha)$, and for a singleton $\alpha = \{i\}$, we abbreviate $A(i) := A(\{i\})$. For any $\alpha \subseteq N$, we denote by $\alpha^c := N \setminus \alpha$ its complement.

If $A[\alpha]$ is invertible, the {Schur complement of $A[\alpha]$ in $A$} is defined by
\begin{equation*}
A / A[\alpha] := A[\alpha^c] - A[\alpha^c, \alpha]\, A[\alpha]^{-1}\, A[\alpha, \alpha^c].
\end{equation*}
In this case, $A$ is invertible if and only if $A / A[\alpha]$ is invertible, and
\begin{equation*}
\det(A) = \det(A[\alpha])\, \det(A / A[\alpha]).
\end{equation*}

\section{Extended Hadamard’s Inequality}

In this section, we derive a sharp lower bound for $\det(A)$ under the assumption that $A$ is in $S^{+}_{n,n-1}$. This result provides a lower-bound analog of the classical Hadamard’s inequality. The following example demonstrates that such a lower bound does not exist when $n=2$. Consequently, meaningful lower bounds can only be expected in dimensions $n \ge 3$.

\begin{ex}
Clearly the matrix
\[
A_t=
\begin{pmatrix}
1 & t\\
t & 1
\end{pmatrix},
\,\,\, t>1
\]
is in $S^{++}_{2,1}$.  However, $\det(A_t)\to -\infty$ as  $t\to\infty$.
\end{ex}

If $A \in S^+_{n,n-1}$ satisfies $a_{ii} = 1$ for all $i$, then $|a_{ij}| \le 1$ for all $i,j$. Another useful property is given in the following lemma.

\begin{lem}
\label{2026-02-02-lem}
Any $A  \in S^+_{n,n-1}$  {has one negative eigenvalue} and all other eigenvalues are nonnegative.
{If $A  \in S^{++}_{n,n-1}$, then the nonnegative eigenvalues are positive.}

\end{lem}

\begin{proof}
Let $\lambda_1 \le \lambda_2 \le \cdots \le \lambda_n$
denote the eigenvalues of $A$.
By assumption $A(i)$ is positive semidefinite, for every $i$, and hence all its eigenvalues are nonnegative. By Cauchy's eigenvalue interlacing theorem (see, for example, \cite[Chapter 4]{HJ1}), we have
$\lambda_2 \ge \lambda_{\min}(A(i)) \ge 0,$
showing that $A$ has at most one negative eigenvalue. {Since}
$\det(A)=\lambda_1 \cdots \lambda_n <0,$$A$ has exactly one negative eigenvalue while the remaining are nonnegative. 
\end{proof}

The following example introduces a class of matrices in $S^+_{n,n-1}$ that will be useful for illustrating the sharpness of the lower bound on the determinant.

\begin{ex}\label{specialapsd}
Consider the $n \times n$ matrix
\[
A= \left(
\begin{array}{cccc}
1 & -x & \cdots & -x \\
-x & 1 & \cdots & -x \\
\vdots & \vdots & \ddots & \vdots \\
-x & -x & \cdots & 1 \\
\end{array}
\right),
\,\,\, 0 < x < 1.
\]
Then $A \in S^{+}_{n,n-1}$, whenever $x \le \frac{1}{n-2}$. Moreover, when $x = \frac{1}{n-2}$, we have
\[
\det(A) = -\frac{1}{n-2}\left(\frac{n-1}{n-2}\right)^{n-1}.
\]
\end{ex}

\begin{thm}
\label{2026-04-04-thm1}
Let $A \in S^+_{n,n-1}$, with $n \ge 3,$ then 
 \begin{align}
 \label{2026-04-14-Had}
 \det(A)\geq-\frac{1}{n-2}\left(\frac{n-1}{n-2}\right)^{n-1}a_{11}\cdots a_{nn}.
\end{align}
{The inequality is strict when $A \in S^{++}_{n,n-1}$.} 
\end{thm}

\begin{proof}
{Let $n \ge 3$. The proof is divided into three steps.} 

{  {\it Step 1.} Suppose $A \in S^{++}_{n,n-1}$ is such that $a_{ii} = 1$ for all $i$.} 
By Lemma~\ref{2026-02-02-lem}, the eigenvalues of $A$ satisfy 
$\lambda_1 < 0  < \lambda_2 \le \cdots \le \lambda_n$ and 
\begin{equation}\label{trace}
\lambda_1 + \cdots + \lambda_n =\operatorname{tr}(A)=n,
\end{equation}
since all diagonal entries of $A$ are equal to $1$. Define the elementary symmetric polynomial of degree $n-1$ in the eigenvalues by
\[
a_{n-1}(\lambda)
=
\sum_{i=1}^n \prod_{j\ne i}\lambda_j.
\]
It is well-known that
\[
a_{n-1}(\lambda)=\sum_{i=1}^n \det(A(i)).
\]
Indeed, the characteristic polynomial of $A$ is given by
\[
\det(tI-A)
=
\prod_{j=1}^n (t-\lambda_j)
=
t^n-a_1(\lambda)t^{n-1}+\cdots+(-1)^n a_n(\lambda).
\]
The coefficient of $t$ is $(-1)^{n-1}a_{n-1}(\lambda)$ and that same coefficient is also the signed sum of the $(n-1)\times (n-1)$ principal minors. Since all proper principal submatrices of $A$ are positive definite, we see that
\begin{align*}
a_{n-1}(\lambda)>0.
\end{align*}

Next we obtain a lower bound on the negative eigenvalue of $A$. Since
\[
\prod_{j\ne i}\lambda_j=\frac{\det(A)}{\lambda_i}, \mbox{ for all $i=1,\ldots,n$},
\]
we have
\begin{align*}
a_{n-1}(\lambda)
=
\det(A)\sum_{i=1}^n \frac{1}{\lambda_i}.
\end{align*}

Since $\det(A)<0$ and $a_{n-1}(\lambda)>0$, it follows that
\begin{align*}
\sum_{i=1}^n \frac{1}{\lambda_i}<0.
\end{align*}

Setting $t:= - \lambda_1>0$, the last inequality becomes
\begin{equation}\label{t}
\sum_{i=2}^n \frac{1}{\lambda_i}<\frac{1}{t}.
\end{equation}

Applying the Cauchy--Schwarz inequality to the vectors 
$$
(\sqrt{\lambda_2},\dots, \sqrt{\lambda_n}) \,\,\,  \mbox{ and } \,\,\, (1/\sqrt{\lambda_2},\dots, 1/\sqrt{\lambda_n}),
$$ 
gives
\[
(n-1)^2
\le
\left(\sum_{i=2}^n\lambda_i\right)
\left(\sum_{i=2}^n\frac{1}{\lambda_i}\right).
\]

From Equation~\eqref{trace}, we have $\lambda_2+\cdots + \lambda_n=n+t$ and thus
\[
\sum_{i=2}^n\frac{1}{\lambda_i}
\ge
\frac{(n-1)^2}{n+t}.
\]

Combining with Equation~\eqref{t}, leads to
\[
\frac{1}{t}>\frac{(n-1)^2}{n+t}
\quad \mbox{or} \quad
n+t>t(n-1)^2
\quad \mbox{or} \quad
t<\frac{1}{n-2}.
\]

Thus,
\begin{equation}\label{smallesteig}
\lambda_1>-\,\frac{1}{n-2}.
\end{equation}

Next, by the arithmetic-geometric mean inequality applied to the positive eigenvalues $\lambda_2,\dots,\lambda_n$, we have
\[
\prod_{i=2}^n\lambda_i
\le
\left(\frac{\lambda_2 + \cdots + \lambda_n}{n-1}\right)^{n-1}
=
\left(\frac{n-\lambda_1}{n-1}\right)^{n-1}.
\]

Multiplying by $\lambda_1<0$ reverses the inequality, and hence
\[
\det(A)
=
\lambda_1\prod_{i=2}^n\lambda_i
\ge
\lambda_1\left(\frac{n-\lambda_1}{n-1}\right)^{n-1}.
\]

Since the function $x \to x (n-x)^{n-1}$ is strictly increasing 
on $(-\infty, 0)$, using \eqref{smallesteig} yields
\[
\det(A)
>
-\frac{1}{n-2}
\left(\frac{n+\frac{1}{n-2}}{n-1}\right)^{n-1}
=
-\frac{1}{n-2}\left(\frac{n-1}{n-2}\right)^{n-1}.
\]

{  {\it Step 2.} Suppose $A \in S^{++}_{n,n-1}$}.
Let $D=\operatorname{diag}(a_{11},\dots,a_{nn})$ and let 
$B := D^{-1/2} A D^{-1/2}.$
Then $B$ has ones on the diagonal, and
$\det(B)=\det(A)/(a_{11}\cdots a_{nn})$. Since $A \in S^{++}_{n,n-1}$, so is $B$. Applying Step 1 to $B$, we conclude that \eqref{2026-04-14-Had} holds with strict inequality.

{  {\it Step 3.} Suppose $A \in S^{+}_{n,n-1}$}. For $\varepsilon>0$, define $A_\varepsilon:=A+\varepsilon I_n.$ Then every proper principal submatrix of $A_\varepsilon$ is positive definite. Since the map $\varepsilon\mapsto \det(A+\varepsilon I_n)$
is continuous and $\det(A) < 0$, one sees that $\det(A_\varepsilon)<0$ for all $\epsilon > 0$ close to zero.
Applying Step 2 to $A_\varepsilon$, we obtain
\[
\det(A_\varepsilon)
>
-\frac{1}{n-2}\left(\frac{n-1}{n-2}\right)^{n-1}
(a_{11}+\varepsilon)\cdots (a_{nn}+\varepsilon).
\]
Taking the limit as $\epsilon \to 0^+$ proves the theorem.
\end{proof}

\begin{ex}
\label{2026-04-04-ex}    
The lower bound in \eqref{2026-04-14-Had} can be approached with matrices in $S^{++}_{n,n-1}$ as follows.
Consider
\[
A(r):=(1-r)I+r\mathbf 1\mathbf 1^\top,
\quad \mbox{where } -\frac{1}{n-2}<r<-\frac{1}{n-1}.
\]

For  \(1 \le k \le n-1\), the eigenvalues of every \(k\times k\) principal submatrix of \(A(r)\) are
$$
\left\{
\begin{array}{ll}
1-r & \text{with multiplicity } k-1 \text{ and } \\
1+(k-1)r & \text{with multiplicity } 1.
\end{array}
\right.
$$
Since \(1-r>0\) and \(k\le n-1\), we have $1+(k-1)r \ge 1+(n-2)r>0$.
Hence every proper principal submatrix of \(A(r)\) is positive definite.
The eigenvalues of \(A(r)\) are
$$
\left\{
\begin{array}{ll}
1-r & \text{with multiplicity } n-1 \text{ and } \\
1+(n-1)r & \text{with multiplicity } 1.
\end{array}
\right.
$$
Because \(r<- 1/(n-1)\), we have \(1+(n-1)r<0\), and therefore
\[
\det(A(r))=(1-r)^{\,n-1}\bigl(1+(n-1)r\bigr)<0.
\]

Moreover,
\[
\det(A(r))
=(1-r)^{\,n-1}\bigl(1+(n-1)r\bigr)
\longrightarrow
-\frac{1}{n-2}\left(\frac{n-1}{n-2}\right)^{n-1}
\qquad \text{as } r\to -\frac{1}{n-2}^{+}.
\]
In the limit, we obtain the matrix  $A\in S^+_{n,n-1}$ with entries

    \[ a_{ij} = \begin{cases} 
          \;\;1 & \text{if} \;\; i=j \\
          -\frac{1}{n-2} & \text{if} \;\; i\neq j
       \end{cases}
    \] 
    and determinant 
    $$
    \det(A)=-\frac{1}{n-2}\left(\frac{n-1}{n-2}\right)^{n-1}.
    $$
\end{ex}
\subsection{A lower bound when the proper leading minors are positive}

This subsection gives a lower bound on the determinant of $A$ for a class of matrices that is  larger than $S^+_{n,n-1}$.

\begin{thm}
\label{2026-04-04-thm-wider}
Let
\begin{align}
\label{2026-02-02-matr}
A=\begin{pmatrix}
B & b\\
b^T & a_{nn}
\end{pmatrix}\in\mathcal S^n~~
\mbox{with $n\geq 3$}, 
\end{align}
where $B$ is an $(n-1) \times (n-1)$ positive semidefinite matrix. Assume that
\begin{align}
\label{2026-04-14-cond1}
{\det(A) \le0, \,\,\, a_{nn} \ge 0, \text{ and } \,\,
|b_i|\le \sqrt{a_{ii}a_{nn}},\quad i=1,\dots,n-1.}
\end{align}
Then
\begin{align}
\label{2026-02-02-lb2}
\det(A)\ge
-(n-1)\left(\frac{n-1}{n-2}\right)^{n-2}\,a_{11}\cdots a_{nn}.
\end{align}
The inequality is strict when $B$ is positive definite and \eqref{2026-04-14-cond1} holds with $a_{nn} >0$.
\end{thm}

\begin{proof}
If $a_{nn}=0$, then the condition $|b_i|\le \sqrt{a_{ii}a_{nn}}$ implies $b_i=0$ for all $i$, hence $b=0$. Thus, $\det(A)=0$ and  both sides of \eqref{2026-02-02-lb2} are zero. If $\det(A)=0$, then again \eqref{2026-02-02-lb2}  holds trivially.  Assume for the rest of the proof that $\det(A) < 0$ and $a_{nn}>0$.

Suppose first that $B$ is an $(n-1) \times (n-1)$ positive definite matrix with diagonal entries equal to $1$.
 The Schur complement of $B$ in $A$ is $a_{nn}-b^TB^{-1}b.$ Thus, 
 $$
 \det(A)=\det(B)(a_{nn}-b^TB^{-1}b).
 $$ 
 Since $\det(B)>0$ and $\det(A) <0,$ it follows that $a_{nn}-b^TB^{-1}b<0.$ Applying the Cauchy–Schwarz inequality, we get \begin{eqnarray*} 
   0 < b^TB^{-1}b=\langle B^{-1}b, b\rangle
    \leq  \lVert B^{-1}b \rVert \lVert{b} \rVert \le \lVert B^{-1}\rVert \lVert b\rVert ^2 
    \le \lVert B^{-1}\rVert (n-1)a_{nn},
    \end{eqnarray*}
    where the last inequality holds since $|b_i| \le \sqrt{a_{nn}}$ for all $i$. 
    Let $\lambda(B)=\{\lambda_1, \dots, \lambda_{n-1}\}$ with $0<\lambda_1\leq \dots \leq \lambda_{n-1}.$ Then $\lVert B^{-1}\rVert=1/\lambda_1$ and we have 
    \begin{eqnarray*}
       \det(A)&=&\det(B)(a_{nn}-b^TB^{-1}b)\\
       &\geq & \det(B)\big(a_{nn}-\lVert B^{-1}\rVert (n-1)a_{nn}\big)\\
       &\geq & \lambda_1\cdots \lambda_{n-1}a_{nn}-\lambda_2\cdots \lambda_{n-1}(n-1)a_{nn} \\
       & > & -\lambda_2\cdots \lambda_{n-1}(n-1)a_{nn},
    \end{eqnarray*}
as $a_{nn} > 0$.
 Thus, minimizing $\det(A)$ is equivalent to maximizing $\lambda_2\cdots \lambda_{n-1}.$ The 
 arithmetic-geometric mean inequality ensures that
\begin{eqnarray*}
  \lambda_2\cdots \lambda_{n-1}
    \leq \left(\frac{\lambda_2+\dots+\lambda_{n-1}}{n-2}\right)^{n-2}
  \leq \left(\frac{\text{tr}(B)}{n-2}\right)^{n-2}
  = \left(\frac{n-1}{n-2}\right)^{n-2},
\end{eqnarray*}
which gives \eqref{2026-02-02-lb2} with strict inequality.

Next, assume that $B$ is a positive definite matrix with arbitrary diagonal entries. Define
$D := \operatorname{diag}(a_{11},\dots, a_{n-1,n-1}, 1)$ and apply the 
first part of the proof to the matrix $\widetilde A := D^{-1/2} A D^{-1/2}$. 
Using the fact that  $\det (\widetilde A) = \det(A)/(a_{11} \cdots a_{n-1,n-1})$, gives \eqref{2026-02-02-lb2} with strict inequality.

Finally, assume $B$ is positive semidefinite.
For $\varepsilon>0$, define $A_\varepsilon := A + \varepsilon I_n.$
Then, the leading $(n-1) \times (n-1)$ principal submatrix of $A_\varepsilon$ is $B+\varepsilon I_{n-1}\succ 0$. Also,
\begin{align*}
|b_i|
\le \sqrt{a_{ii}a_{nn}}
\le \sqrt{(a_{ii}+\varepsilon)(a_{nn}+\varepsilon)} \mbox{ for each $i=1,\dots,n-1$}.
\end{align*}
Since $\det(A_\varepsilon)\to \det(A) < 0$ as $\varepsilon$ approaches $0$, we have
$\det(A_\varepsilon)<0$ for all sufficiently small $\varepsilon>0$. Applying the last paragraph to $A_\varepsilon$ gives
\begin{align*}
\det(A_\varepsilon)>
-(n-1)\left(\frac{n-1}{n-2}\right)^{n-2}
(a_{11}+\varepsilon)\cdots(a_{nn}+\varepsilon).
\end{align*}
Taking the limit as $\varepsilon \to 0^+$, completes the proof.
\end{proof}

\begin{ex}
The condition $|b_i|\le \sqrt{a_{ii}a_{nn}},$ for  $i=1,\dots,n-1$,
is essential. Without it the conclusion fails even for $n=3$. Indeed, consider
\begin{align*}
A=
\begin{pmatrix}
1 & 0 & t\\
0 & 1 & 0\\
t & 0 & 1
\end{pmatrix},\,\,\, t>1.
\end{align*}
Then $B=I_2\succ 0$, $a_{33}=1>0$, and
$\det(A)=1-t^2 \to -\infty$ as $t\to\infty$. On the other hand, the theorem would give
\begin{align*}
\det(A)> {-4},
\end{align*}
a contradiction for {$t>\sqrt{5}$}.
\end{ex}

As one should expect, the lower bound in Theorem~\ref{2026-04-04-thm-wider} is smaller than the lower bound in Theorem~\ref{2026-04-04-thm1}. The two lower bounds are equal only when $n=3$.

\begin{ex}
The lower bound  \eqref{2026-02-02-lb2} can be attained.
Indeed, let  $A$ be of the form \eqref{2026-02-02-matr} with $a_{nn}=1$, $b=(1,\dots ,1)^T$, and
 where $B=(b_{ij})$ is defined as
    \[ b_{ij} = \begin{cases} 
          \;\;1 & \text{if} \;\; i=j, \\
          -\frac{1}{n-2} & \text{if} \;\; i\neq j.
       \end{cases}
    \]
It can be verified that  every proper principle minor of $B$ is positive, $\det(B)=0$, and for this $A$, inequality \eqref{2026-02-02-lb2} holds with equality.
\end{ex}

\section{Extended Fischer’s inequality}

Fisher’s inequality~\eqref{eq:fisher} provides an upper bound for $\det(A)$ when $A$ is positive semidefinite.  In this section, we establish an analogous  sharp lower bound for matrices in $S^+_{n,n-1}$. 
 
%

\begin{thm}
\label{2026-04-05-thm}
Let $A \in S^+_{n,n-1}$, with $n \ge 3$, and let $\alpha \subseteq \{1,\ldots,n\}$. Then
\begin{align}
\label{2026-04-15-ineq}
\det(A) \ge -\frac{1}{n-2}\left(\frac{n-1}{n-2}\right)^{n-1}
\det (A[\alpha])\det (A[\alpha^c]).
\end{align}
{The inequality is sharp when $A \in S^{++}_{n,n-1}$.}
\end{thm}
\begin{proof}
{Suppose first that $A \in S^{++}_{n,n-1}$.}  Let $A[\alpha]=U\Lambda U^T$ and $A[\alpha^c]=V\Gamma V^T$ represent the spectral decompositions of $A$ and $B$, respectively. Note that $U$ and $V$ are unitary matrices and $\Lambda={\rm diag} \{\lambda_1,\dots,\lambda_p\}$, 
$\Gamma={\rm diag} \{\nu_1,\dots, \nu_q \}$   are positive diagonal matrices, where $p := |\alpha|$ and $q:=|\alpha^c|$. Let $W:=U\oplus V$ denote the block diagonal matrix formed by $U$ and $V.$ The expression for $W^TAW$ can be computed as: 
$$
W^TAW=\begin{pmatrix}
    \Lambda & U^T A[\alpha, \alpha^c] U\\
    V^T A[\alpha, \alpha^c]^T V & \Gamma
\end{pmatrix}.
$$  
Theorem~\ref{2026-04-04-thm1} ensures that 
\begin{eqnarray*}
 \det(A)&=&\det(W^TAW) 
 > -\frac{1}{n-2}\left(\frac{n-1}{n-2}\right)^{n-1}{(\lambda_1\cdots\lambda_p )(\nu_1\cdots \nu_q )}\\
 &=& -\frac{1}{n-2}\left(\frac{n-1}{n-2}\right)^{n-1}\det(A[\alpha]) \det {(A[\alpha^c])}.
\end{eqnarray*}

{Suppose $A \in S^{+}_{n,n-1}$.} 
For \(\varepsilon>0\), define $A_\varepsilon:=A+\varepsilon I_n.$
Since every proper principal submatrix of \(A\) is positive semidefinite,
it follows that every proper principal submatrix of \(A_\varepsilon\) is
positive definite. Since the map $\varepsilon\mapsto \det(A+\varepsilon I_n)$
is continuous and $\det(A)<0$, we have $\det(A_\varepsilon)<0$
for all $\epsilon > 0$ close to zero. Applying the first paragraph to \(A_\varepsilon \in S^{++}_{n,n-1}\), we obtain
\[
\det(A_\varepsilon)>
-\frac{1}{n-2}\left(\frac{n-1}{n-2}\right)^{n-1}
\det (A_\varepsilon[\alpha]) \det (A_\varepsilon[\alpha^c]).
\]
Taking the limit as $\epsilon \to 0^+$ proves the theorem.
\end{proof}

%
%
%

\begin{ex}[Sharpness of Theorem~\ref{2026-04-05-thm}]
\label{2026-04-16-ex}
Let \(n \ge 3\) and let \(\alpha=\{n\}\). Define
\[
p:=\left(\frac{n-2}{n-1}\right)^{n-1}, 
\qquad
t:=\frac{1-p}{n-1},
\qquad
s^2:=\frac{1-t(n-2)}{n-2}
=\frac{1+(n-2)p}{(n-1)(n-2)}.
\]
Set
\[
B:=I_{n-1}-tJ_{n-1},
\]
and consider the matrix
\[
A=
\begin{pmatrix}
B & s\mathbf{1}\\
s\mathbf{1}^\top & 1
\end{pmatrix}
\in\mathcal S^n,
\]
where \(\mathbf{1}\in\mathbb R^{\,n-1}\) is the all-ones vector. We claim that $A \in S^+_{n,n-1}$ and that \eqref{2026-04-15-ineq} holds with equality.

Indeed, note that $B$ has eigenvalues
$1$  with multiplicity $n-2$, and $1 - t(n-1) = p > 0$, so $B$ is positive definite. 
Hence, the principal submatrix obtained by deleting the {last} row and column is positive definite.

Now let \(i\in\{1,\dots,n-1\}\). Deleting the \(i\)th row and column of \(A\) yields a principal \((n-1)\times(n-1)\) submatrix permutation-similar to
\[
M=
\begin{pmatrix}
I_{n-2}-tJ_{n-2} & s\mathbf{1}\\
s\mathbf{1}^\top & 1
\end{pmatrix}.
\]
Since \(I_{n-2}-tJ_{n-2}\succ0\),  the Schur complement of \(I_{n-2}-tJ_{n-2}\) in \(M\) is
\[
1-s^2\mathbf{1}^\top (I_{n-2}-tJ_{n-2})^{-1}\mathbf{1}
=
1-s^2\frac{n-2}{1-t(n-2)}
=0,
\]
by the definition of \(s^2\). Hence \(M\succeq0\). Therefore every principal \((n-1)\times(n-1)\) submatrix of \(A\) is positive semidefinite, and so $A\in S^+_{n,n-1}.$ Next, since \(\alpha=\{n\}\), we have
$\det(A[\alpha])=1$ and $\det(A[\alpha^c])=\det(B)=p.$

Finally, by the block determinant formula,
\[
\det(A)
=\det(B)\Bigl(1-s^2\mathbf{1}^\top B^{-1}\mathbf{1}\Bigr).
\]
Since \(B\mathbf{1}=p\mathbf{1}\), we get
\[
\mathbf{1}^\top B^{-1}\mathbf{1}=\frac{n-1}{p}.
\]
Therefore
\[
\det(A)
=
p\left(1-\frac{(n-1)s^2}{p}\right) = p - \frac{1+(n-2)p}{n-2} = -\frac{1}{n-2}.
\]
Recalling the definition of $p$, it follows that
\[
\det(A)
=
-\frac{1}{n-2}
\left(\frac{n-1}{n-2}\right)^{n-1}
\det(A[\alpha])\det(A[\alpha^c]).
\]
Thus \eqref{2026-04-15-ineq} holds with equality.
\end{ex}


\section{Extended Koteljanskii Inequality}

In this section, we derive a Koteljanskii-type inequality for matrices in $S^{+}_{n,n-1}$. The classical Koteljanskii's inequality, \eqref{eq:koteljanskii}, holds for positive semidefinite matrices. However, when $\det (A[\alpha\cup\beta]) < 0$ and all proper principal submatrices are positive definite, the left-hand side of \eqref{eq:koteljanskii} becomes negative. Consequently, although inequality~\eqref{eq:koteljanskii} remains valid, it is no longer informative in this setting.
The next result provides a suitable extension of inequality~\eqref{eq:koteljanskii}.

\begin{thm}[Extended Koteljanskii inequality]
\label{Extended Koteljanskii inequality}
Let $A \in \mathcal{S}^n$, and let $\alpha, \beta \subseteq \{1,\dots,n\}$, with $r:=|(\alpha \cup \beta) \setminus (\alpha \cap \beta)| \ge 3$ {and $m:=|\alpha \cup \beta|$. If $A[\alpha \cup \beta] \in S^{+}_{m,m-1}$}, then
\begin{equation}\label{ext-kot-equiv}
\det (A[\alpha \cup \beta]) \det (A[\alpha \cap \beta])
\ge 
-\frac{1}{r-2}\left(\frac{r-1}{r-2}\right)^{r-1}
\det (A[\alpha]) \det (A[\beta]).
\end{equation}
{The inequality is sharp when $A[\alpha \cup \beta] \in S^{++}_{m,m-1}$.}
\end{thm}

\begin{proof}
{Suppose that $A[\alpha \cup \beta] \in S^{++}_{m,m-1}$.}
Let $\gamma = \alpha \cap \beta,  \omega = \alpha \cup \beta$, 
$\mu=\alpha\setminus\beta$, and  $\nu=\beta\setminus\alpha$.
Then $\omega=\gamma\cup\mu\cup\nu,$
where \(\gamma,\mu,\nu\) are pairwise disjoint, and $r=|\omega\setminus\gamma|=|\mu|+|\nu|.$

By hypothesis, \(\det {(A[\omega])}<0\), and every proper principal submatrix of
\(A[\omega]\) is positive definite. In particular, since \(\gamma\subsetneq \omega\), the
principal submatrix \(A[\gamma]\) is positive definite. Therefore,  the Schur complement of
\(A[\gamma]\) in \(A[\omega]\) is well-defined. 

Write \(A[\omega]\) in block form according to the partition
\(\omega=\gamma\cup (\mu\cup\nu)\):
\[
A[\omega]=
\begin{pmatrix}
A[\gamma] & A[\gamma,\mu\cup\nu]\\
A[\gamma,\mu\cup\nu]^T & A[\mu\cup\nu]
\end{pmatrix}.
\]
Let
\[
S:=A[\omega]/A[\gamma]
=
A[\mu\cup\nu]-A[\gamma,\mu\cup\nu]^T\,A[\gamma]^{-1}\,A[\gamma,\mu\cup\nu] \in {\cal S}^r,
\]
since \(|\mu\cup\nu|=r\).  Using the well-known determinant formula for Schur complements, we have 
\begin{equation*}\label{det-schur}
\det {(S)}=\frac{\det {(A[\omega])}}{\det {(A[\gamma])}}.
\end{equation*}
Since \(A[\gamma]\) is positive definite, and 
\(\det {(A[\omega])}<0\), it follows that
\[
\det {(S)}<0.
\]

We now relate principal submatrices of \(S\) to Schur complements of
principal submatrices of \(A\). We use the quotient property of Schur
complements due to Crabtree and Haynsworth
\cite{CrabtreeHaynsworth1969}: if \(B\) is a nonsingular principal
submatrix of \(A\) and \(C\) is a nonsingular principal submatrix of
\(B\), then \(B/C\) is a principal submatrix of \(A/C\) and
\[
A/B=(A/C)/(B/C).
\]

Since $A[\gamma]$ is nonsingular, we may apply the quotient property for Schur complements to the matrices
\[
A \to A[\omega], \qquad B = A[\gamma\cup\tau], \qquad C = A[\gamma],
\]
where $\tau \subseteq \mu \cup \nu$. This yields
\begin{equation*}
A[\omega]/A[\gamma\cup\tau]
=
\bigl(A[\omega]/A[\gamma]\bigr)
\Big/
\bigl(A[\gamma\cup\tau]/A[\gamma]\bigr).
\end{equation*}

Moreover, the quotient property also implies that $A[\gamma\cup\tau]/A[\gamma]$ is a principal submatrix of $A[\omega]/A[\gamma]$. Indeed, after eliminating the indices in $\gamma$ from $\gamma\cup\tau$, the remaining index set is precisely $\tau$. Therefore,
\[
A[\gamma\cup\tau]/A[\gamma]
=
\bigl(A[\omega]/A[\gamma]\bigr)[\tau] =
S[\tau],
\]
recalling that $S = A[\omega]/A[\gamma]$. 
Applying the last equation with \(\tau=\mu\) and \(\tau=\nu\), and
noting that \(\gamma\cup\mu=\alpha\) and \(\gamma\cup\nu=\beta\),
gives
\[
S[\mu]=A[\alpha]/A[\gamma],\qquad
S[\nu]=A[\beta]/A[\gamma].
\]
Taking determinants,
\begin{align}
\label{2026-04-15-ident}
\det {(S[\mu])}=\frac{\det {(A[\alpha])}}{\det {(A[\gamma])}},
\qquad
\det {(S[\nu])}=\frac{\det {(A[\beta])}}{\det {(A[\gamma])}}.
\end{align}

Next, we show that every proper principal submatrix of \(S\) is positive
definite. Let \(\tau\subsetneq\mu\cup\nu\). Then
\[
S[\tau]=A[\gamma\cup\tau]/A[\gamma].
\]
Since \(\gamma\cup\tau\subsetneq\omega\), the matrix \(A[\gamma\cup\tau]\)
is a proper principal submatrix of \(A[\omega]\) and hence is positive
definite by assumption. It is well-known that the Schur complement with respect to  any principal submatrix of a positive definite
matrix remains positive definite.
Thus, every proper principal submatrix of \(S\) is positive definite.

Partition \(S\) according to \(\mu\cup\nu\):
\[
S=
\begin{pmatrix}
S[\mu] & S[\mu,\nu]\\
S[\nu,\mu] & S[\nu]
\end{pmatrix}.
\]

Since {\(S\in{\cal S}^{++}_{r,r-1}\), \(r\ge 3\), Theorem~\ref{2026-04-05-thm}}
implies that
\[
\det {(S)}>
-\frac{1}{r-2}
\left(\frac{r-1}{r-2}\right)^{r-1}
\det {(S[\mu])}\det {(S[\nu])}.
\]

Substituting identities \eqref{2026-04-15-ident} above, gives
\[
\frac{\det {(A[\omega])}}{\det {(A[\gamma])}}
>
-\frac{1}{r-2}
\left(\frac{r-1}{r-2}\right)^{r-1}
\frac{\det {(A[\alpha])}}{\det {(A[\gamma])}}
\frac{\det {(A[\beta])}}{\det {(A[\gamma])}}.
\]
Multiplying both sides by \((\det {(A[\gamma])})^2>0\) gives {the strict inequality in} \eqref{ext-kot-equiv}.

{Suppose  $A[\alpha \cup \beta] \in S^{+}_{m,m-1}$.}  For $\varepsilon>0$, define $A_\varepsilon:=A+\varepsilon I_n.$
Then $A_\varepsilon[\omega]=A[\omega]+\varepsilon I_{|\omega|}.$
Since every proper principal submatrix of $A[\omega]$ is positive semidefinite,
every proper principal submatrix of $A_\varepsilon[\omega]$ is positive definite.
Moreover, since  and the map 
$\varepsilon\mapsto \det (A[\omega]+\varepsilon I_{|\omega|})$
is continuous and $\det (A[\omega])<0$, we have 
$\det (A_\varepsilon[\omega])<0$ for all  $\varepsilon> 0$ close to zero. 
Applying the positive definite version to $A_\varepsilon$, we obtain
\[
\det (A_\varepsilon[\alpha \cup \beta]) \det (A_\varepsilon[\alpha \cap \beta])
>
-\frac{1}{r-2}\left(\frac{r-1}{r-2}\right)^{r-1}
\det (A_\varepsilon[\alpha]) \det (A_\varepsilon[\beta]).
\]
Taking the limit as $\epsilon \to 0^+$ proves the theorem.
\end{proof}

\begin{ex}
We now present an explicit example illustrating Theorem~\ref{Extended Koteljanskii inequality}. 
Let
\[
A=(1-c)I_6+cJ_6,
\mbox{ where } c=-\frac14,
\]
where \(J_6\) denotes the \(6\times 6\) all-ones matrix. Explicitly,
\[
A=
\begin{pmatrix}
1 & -\frac14 & -\frac14 & -\frac14 & -\frac14 & -\frac14\\
-\frac14 & 1 & -\frac14 & -\frac14 & -\frac14 & -\frac14\\
-\frac14 & -\frac14 & 1 & -\frac14 & -\frac14 & -\frac14\\
-\frac14 & -\frac14 & -\frac14 & 1 & -\frac14 & -\frac14\\
-\frac14 & -\frac14 & -\frac14 & -\frac14 & 1 & -\frac14\\
-\frac14 & -\frac14 & -\frac14 & -\frac14 & -\frac14 & 1
\end{pmatrix}.
\]

For a matrix of the form \((1-c)I_m+cJ_m\), the eigenvalues are
\[
1-c \quad \text{with multiplicity } m-1,
\qquad
1+(m-1)c \quad \text{with multiplicity } 1.
\]
Hence the eigenvalues of \(A\) are
\[
\frac54 \quad \text{with multiplicity } 5,
\qquad
1+5\!\left(-\frac14\right)=-\frac14 \quad \text{with multiplicity } 1.
\]
Therefore
\[
\det (A)<0.
\]

Now let \(B\) be any proper principal submatrix of \(A\), say of order \(k\le 5\).
Then \(B\) has the same form \((1-c)I_k+cJ_k\), and hence its eigenvalues are
\[
\frac54 \quad \text{with multiplicity } k-1,
\qquad
1+(k-1)\!\left(-\frac14\right)=1-\frac{k-1}{4}
\quad \text{with multiplicity } 1.
\]
Since \(k\le 5\), we have
\[
1-\frac{k-1}{4}\ge 1-\frac44=0.
\]
Moreover, if \(k\le 5\) is proper, then in fact \(k\le 5\) and for \(k<6\) we get
\[
1-\frac{k-1}{4}>0.
\]
Thus every proper principal submatrix of \(A\) is positive definite.

Choose
\[
\alpha=\{1,2,3,4\},
\qquad
\beta=\{3,4,5,6\}.
\]
Then
\[
\alpha\cap\beta=\{3,4\},
\qquad
\alpha\cup\beta=\{1,2,3,4,5,6\},
\]
so that
\[
r=|(\alpha\cup\beta)\setminus(\alpha \cap \beta)|=4.
\]

For any principal submatrix of order \(k\), the determinant is
\[
(1-c)^{k-1}(1+(k-1)c)
=
\left(\frac54\right)^{k-1}\!\left(1-\frac{k-1}{4}\right).
\]
Therefore
\[
\det A[\gamma]
=
\frac54\cdot \frac34
=
\frac{15}{16},
\]
since \(|\alpha\cap\beta|=2\),
\[
\det (A[\alpha])=\det (A[\beta])
=
\left(\frac54\right)^3\!\left(\frac14\right)
=
\frac{125}{256},
\]
since \(|\alpha|=|\beta|=4\), and
\[
\det (A[\alpha\cup\beta])=\det (A)
=
\left(\frac54\right)^5\!\left(-\frac14\right)
=
-\frac{3125}{4096}.
\]

Since \(r=4\), we have
\[
\frac{1}{r-2}\left(\frac{r-1}{r-2}\right)^{r-1}
=
\frac12\left(\frac32\right)^3
=
\frac{27}{16}.
\]
Hence
\[
\det (A[\alpha\cup\beta]) \det (A[\alpha\cap\beta])
=
-\frac{3125}{4096}\cdot \frac{15}{16}
=
-\frac{46875}{65536},
\]
while
\[
-\frac{27}{16}\det (A[\alpha]) \det (A[\beta])
=
-\frac{27}{16}\left(\frac{125}{256}\right)^2
=
-\frac{421875}{1048576}.
\]
Therefore
\[
\det (A[\alpha\cup\beta]) \det (A[\alpha\cap\beta])
>
-\frac{1}{r-2}\left(\frac{r-1}{r-2}\right)^{r-1}
\det (A[\alpha] ) \det (A[\beta]).
\]
Thus \eqref{ext-kot-equiv} holds in this example.
\end{ex}

\begin{ex}[Sharpness]
Let $r \ge 3$, $n=r$, and take
\[
\alpha=\{1\}, \qquad \beta=\{2,\dots,r\}.
\]
Then $\alpha\cap\beta=\varnothing$, $\alpha\cup\beta=\{1,\dots,r\}$, and Theorem
\ref{Extended Koteljanskii inequality} reduces to Theorem~\ref{2026-04-05-thm}.
The extremal matrix from {Example~\ref{2026-04-16-ex}} shows that the constant in Theorem \ref{Extended Koteljanskii inequality} is sharp.
\end{ex}

\section*{Acknowledgments}

S.M.\ Fallat is supported in part by an NSERC Discovery Research Grant, Application No.: RGPIN-2025-05272.
 The work of the PIMS Postdoctoral Fellow S.\ Mondal leading to this publication was supported in part by the Pacific Institute for the Mathematical Sciences.  H.\ Sendov is supported in part by an NSERC Discovery Research Grant, Application No.: RGPIN-2020-06425.

\end{document}